\documentclass[12pt,a4paper]{amsart}
\usepackage{geometry}
\usepackage[utf8]{inputenc}
\usepackage[english]{babel}
\usepackage{amsfonts,amssymb,amsmath,amsthm,mathtools,bm}
\usepackage{xcolor}
\usepackage{enumitem}
\usepackage{booktabs}
\usepackage{cite}
\usepackage{url}
\usepackage{footmisc}
\usepackage[hidelinks,hyperfootnotes=false]{hyperref}
\hypersetup{
  pdftitle={The derivative function of the zeta-star correspondence},
  pdfkeywords={multiple zeta-star values, right derivative,
    nowhere differentiability, Hausdorff dimension, packing dimension,
    multifractal spectrum}
}

\allowdisplaybreaks[2]
\numberwithin{equation}{section}
\theoremstyle{plain}
\newtheorem{theorem}{Theorem}[section]
\newtheorem{corollary}[theorem]{Corollary}

\newtheorem{lemma}[theorem]{Lemma}
\newtheorem{proposition}[theorem]{Proposition}
\newtheorem{remark}[theorem]{Remark}

\newtheorem{conjecture}[theorem]{Conjecture}

\makeatletter
\renewenvironment{proof}[1][\proofname]{\par
  \pushQED{\qed}%
  \normalfont \topsep6\p@\@plus6\p@\relax
  \trivlist
  \item[\hskip\labelsep
        \normalfont\bfseries\color{black}#1\@addpunct{:}]\ignorespaces
}{%
  \popQED\endtrivlist\@endpefalse
}
\makeatother

\newcommand{\Rp}{\partial_{+}}
\newcommand{\dimH}{\dim_{\mathrm H}}
\newcommand{\dimP}{\dim_{\mathrm P}}

\newcommand{\ldimB}{\underline{\dim}_{\mathrm B}}
\newcommand{\udimB}{\overline{\dim}_{\mathrm B}}
\newcommand{\udimMB}{\overline{\dim}_{\mathrm{MB}}}

\newcommand{\Btwo}{B_{\mathbb R^2}}
\newcommand{\eps}{\varepsilon}

\newcommand{\N}{\mathbb N}
\newcommand{\Z}{\mathbb Z}

\newcommand{\SigmaTwo}{\{0,1\}^{\N}}

\newcommand{\R}{\mathbb R}

\newcommand{\CC}{\mathcal C}
\newcommand{\HH}{\mathcal H}
\newcommand{\LL}{\mathcal L^1}

\newtheorem*{mtp}{Mass transference principle}
\title[Derivative of the zeta-star correspondence]{The derivative function of the zeta-star correspondence}

\author{Jiangtao Li}

\email{lijiangtao@csu.edu.cn}

\address{Jiangtao Li \\ School of Mathematics and Statistics, HNP-LAMA, Central South University, Hunan Province, China}

\DeclareMathOperator{\wt}{wt}
\DeclareMathOperator{\dep}{dep}

\begin{document}

\begin{abstract}
We study the right derivative of the zeta-star correspondence between
binary expansions and multiple zeta-star values of infinite depth.  At a
finite-index point, this derivative is a normalized difference of two
finite-depth multiple zeta-star values.  We use this identity to express its
local behavior in terms of lower-truncated star sums and their boundary
coefficients.  The homogeneous contribution at summation variable three
contracts with ratio $2/3$; the full coefficient also receives a positive
forcing term.  A weighted sum of the remaining coefficients contracts
by at most $1/2$ at each binary step.

On the open parameter interval, we prove that the derivative is
right-continuous everywhere and continuous precisely at the non-dyadic points.  Its left jump at a dyadic point is given
by an explicit positive multiple series, with divergence exactly at
$1/2-2^{-M}$, $M\ge2$.  A refinement of the index gives an infinite upper
right Dini derivative at every interior point.  We also prove that the
Hausdorff, packing, and modified upper box dimensions of the graph are all
$\log_2(8/3)$.  The Hausdorff lower bound uses the dimension of the Bernoulli
convolution with parameter $2/3$, together with a separate measure-transfer
argument.  Finally, the pointwise H\"older exponent at a non-dyadic point is
$\log_2(3/2)$ divided by its dyadic approximation exponent, and the
corresponding Hausdorff spectrum is linear.  
\end{abstract}

\maketitle

\begingroup
\let\thefootnote\relax
\footnotetext{%
}
\endgroup

\section{Introduction}\label{sec:main}

Multiple zeta-star values can be studied not only as individual special
values, but also as an ordered family whose finite depth elements
approximate a continuum of infinite depth limits.  Zeta star
correspondence makes this second viewpoint precise by associating binary
coordinates with those limits.  Its derivative therefore measures how
arithmetic values change under refinement of their indices.  The purpose
of this paper is to determine the fine regularity of that derivative:
its one-sided limits, the failure of further differentiability, the
dimensions of its graph, and the distribution of its pointwise H\"older
exponents.  These questions connect the arithmetic of multiple harmonic
sums with quantitative approximation and fractal geometry.

\subsection{Multiple zeta-star values and the infinite depth viewpoint}

We use the descending summation convention.  An index
$\mathbf k=(k_1,\ldots,k_r)$ is admissible if $k_1\ge2$ and
$k_2,\ldots,k_r\ge1$.  Its weight and depth are
$\wt(\mathbf k)=k_1+\cdots+k_r$ and $\dep(\mathbf k)=r$.  The multiple zeta
value and multiple zeta-star value attached to $\mathbf k$ are
\begin{equation}\label{eq:mzv-mzsv-intro}
 \zeta(\mathbf k)
 =\sum_{n_1>\cdots>n_r\ge1}
 \frac1{n_1^{k_1}\cdots n_r^{k_r}},
 \qquad
 \zeta^\star(\mathbf k)
 =\sum_{n_1\ge\cdots\ge n_r\ge1}
 \frac1{n_1^{k_1}\cdots n_r^{k_r}}.
\end{equation}
We set $\wt(\varnothing)=\dep(\varnothing)=0$ and
$\zeta(\varnothing)=\zeta^\star(\varnothing)=1$, and write
$\mathbf k^{[j]}=(k_1,\ldots,k_j)$, with
$\mathbf k^{[0]}=\varnothing$.  As usual, $\{a\}^q$ denotes $q$
consecutive copies of $a$.

The algebraic theory of these series provides identities among special
values and organizes them by weight and depth
\cite{Hoffman1992,Hoffman1997,Zagier1994}.  Their relation to multiple
polylogarithms brings together explicit evaluations and applications in
combinatorics and mathematical physics \cite{BBBL2001}, while Brown's
work \cite{Brown2012} on mixed Tate motives establishes that every multiple zeta value is
a rational linear combination of values with entries in $\{2,3\}$.  The star and strict versions are connected by the
familiar sum over coarsenings of an index: each comma may be retained
or replaced by addition of the adjacent entries.  Thus the star notation
does not introduce a different algebra of finite-depth constants. 

A complementary line of research concerns order and limiting depth.
The author constructed an order compatible correspondence between infinite
admissible indices and the half-line $(1,+\infty)$, and studied its
topological and metric properties \cite{LiTopology2023}.
Hirose, Murahara and Onozuka developed infinite-length multiple zeta-star
values, obtained explicit evaluations, and established one-sided
derivative formulas for the associated binary parametrization
\cite{HMO2025}.  In this setting, increasing depth is not merely a device
for producing additional identities: it provides a way to approximate
real numbers by multiple zeta-star values.

We use the notation of \cite{LiTopology2023}. To describe the normalization used here, let
\begin{equation}\label{eq:T-intro}
 \mathcal T=
 \bigl\{(k_1,k_2,\ldots):k_1\ge2,\ k_j\ge1\ (j\ge2)\bigr\}
 \setminus\bigl\{(2,\{1\}^{\infty})\bigr\}.
\end{equation}
We order $\mathcal T$ by reverse lexicographic order:
if $j$ is the first position at which $\mathbf k$ and
$\mathbf l$ differ, then $\mathbf k\succ\mathbf l$
if and only if $k_j<l_j$.
For $\mathbf k\in\mathcal T$, put
\begin{equation}\label{eq:beta-infinite-intro}
 K_r=k_1+\cdots+k_r,
 \qquad
 \beta(\mathbf k)=\sum_{r\ge1}2^{-K_r}.
\end{equation}
The positions of the ones in the binary expansion are exactly the
partial weights $K_r$.  
The author \cite{LiTopology2023} proved that there is a bijective order preserving map 
\[
\eta: \mathcal{T}\rightarrow (1,+\infty),\]
\[
\eta(k_1,k_2,\cdots)=\lim_{r\to+\infty}\zeta^\star(k_1,\cdots,k_r).\]
We call $\eta$ the  zeta-star correspondence.  Hirose, Murahara and Onozuka \cite{HMO2025} found the bijective map $\eta$ independently.
One can identify $\mathcal{T}$ with $(0,1/2)$ by the map $\beta$.
By considering the inverse map of  $\tau=\beta\circ \eta^{-1}$ in Section $4$, \cite{LiTopology2023}, the preceding order results yield a continuous
increasing correspondence
\begin{equation}\label{eq:Phi-intro}
 \Phi:[0,1/2]\rightarrow[1,+\infty],
 \qquad
 \Phi\bigl(\beta(\mathbf k)\bigr)
 =\lim_{r\to+\infty}\zeta^\star(k_1,\cdots,k_r),
\end{equation}
where the target has its order topology,
$\Phi(0)=1$, and $\Phi(1/2)=+\infty$.  For a finite admissible index define
\begin{equation}\label{eq:finite-beta}
 \beta^f(\mathbf k)=\sum_{j=1}^r2^{-K_j}.
\end{equation}
The infinite-one-tail identity of  Lemma $3.4$ $(ii)$ in  \cite{LiTopology2023} gives
\begin{equation}\label{eq:finite-phi}
 \Phi\bigl(\beta^f(\mathbf k)\bigr)=\zeta^\star(\mathbf k).
\end{equation}
Hence the dyadic coordinates in $(0,1/2)$ are precisely the finite index
coordinates.  We recall the identity behind \eqref{eq:finite-phi} in
Section~\ref{sec:truncated}.

This viewpoint has already led to several related developments.
Rational deformations connect the order structure with bounded variation
and Cantor-type sets \cite{LiRational2025}.  Kamano \cite{Kamano2026} investigated the order structure of multiple polylogarithms.
Metric approximation by
multiple zeta-star values and its normalized approximation function are
studied in \cite{Li2025,Li2026}.  Restricted-index arithmetic sums and
products are treated in \cite{LiYang2026}, and complex variations obtained
from finite harmonic-star sums in \cite{LiDivided2026}.

A two-dimensional generalization is the \emph{polylogarithms star
correspondence} introduced by the author \cite{LiPolylogarithms2026}.
For a finite index $\mathbf k=(k_1,\ldots,k_r)$ with positive entries, set
\[
 \widetilde{\operatorname{Li}}^\star_{\mathbf k}(z)
 =\sum_{n_1\ge\cdots\ge n_r\ge1}
   \frac{z^{n_1-1}}{n_1^{k_1}\cdots n_r^{k_r}},
 \qquad |z|<1.
\]
This is the shifted multiple star polylogarithm; boundary values are used
on the unit circle. For $|z|=1$, the infinite depth map
\[
 \eta_S(z,\mathbf k)
 =\lim_{r\to\infty}
   \widetilde{\operatorname{Li}}^\star_{\mathbf k^{[r]}}(z)
\]
is defined for every infinite positive integer index when $z\ne1$, and
for $\mathbf k\in\mathcal T$ when $z=1$.

The author \cite{ LiPolylogarithms2026} proved that $\eta_S$ is a bijection onto the punctured closed
half-plane \[\{w\in\mathbb C:\operatorname{Re}w\ge1/2\}\setminus\{1\}.\]
Extending $\beta$ by the same partial-weight formula to all infinite
positive-integer indices, the coordinate $z\beta(\mathbf k)$ makes this
bijection a homeomorphism. At $z=1$ it recovers
$\eta_S(1,\mathbf k)=\Phi(\beta(\mathbf k))$, so the real correspondence
considered here is a slice of this two-dimensional construction.

These works motivate a finer question about the original real zeta star
correspondence. How stable is its local arithmetic scaling as the binary
coordinate varies?

\subsection{The derivative as an arithmetic regularity problem}

We consider
\begin{equation}\label{eq:D-intro}
 D(y)=\partial_+\Phi(y),\qquad 0<y<1/2.
\end{equation}
The derivative formulas of Hirose--Murahara--Onozuka \cite{HMO2025}, expressed in the
normalization form of Proposition $2.1$ in \cite{Li2026}, give
\begin{equation}\label{eq:finite-D-intro}
 D\bigl(\beta^f(\mathbf k)\bigr)
 =2^{\wt(\mathbf k)}
 \bigl(\zeta^\star(\mathbf k)
       -\zeta^\star(\mathbf k^{[r-1]})\bigr)=\sum_{n_1\geq \cdots \geq n_r\geq 2} \left(\frac{2}{n_1}   \right)^{k_1}\cdots \left(\frac{2}{n_r} \right)^{k_r}
\end{equation}
for \[{\bf k}=(k_1,\cdots,k_r),\quad {\bf k}^{[r-1]}=(k_1,\cdots,k_{r-1}).\]
At a non-dyadic point $y=\beta(\mathbf k)$, these finite index expressions
converge to $D(y)=\Phi'(y)$.  In particular, $D$ is not an auxiliary
function imposed on the arithmetic construction: its finite coordinate
values are normalized differences of ordinary finite-depth star values.
Equivalently, they are normalized star sums with every summation
variable at least two.

The approximation meaning of $D$ is already visible from
differentiability of $\Phi$.  At a non-dyadic $y=\beta(\mathbf k)$,
\[
 \lim_{r\to\infty}
 \frac{\Phi(y)-\zeta^\star(\mathbf k^{[r]})}
      {y-\beta^f(\mathbf k^{[r]})}=D(y).
\]
This is a direct consequence of \eqref{eq:finite-phi} and the known
derivative formula.  It
identifies $D(y)$ as the asymptotic conversion factor from binary
truncation error to error in the corresponding star value.  Understanding
its continuity and oscillation is therefore relevant to the stability
of that conversion, not only to the qualitative differentiability of
$\Phi$.

The existence of a first derivative leaves its own regularity
undetermined.  In particular, two distinctions are essential here.
First, the dyadic difference
$\partial_-\Phi(y)-\partial_+\Phi(y)$ is already present, with
$Z^\star(x)=\Phi(x/2)$, in Theorem $4.10$ and Section $5$, \cite{HMO2025}.  Our one sided limit result
proves the existence of the latter limit directly and identifies its
value, including the divergent cases.  Second, continuity of $D$ at a
point does not imply differentiability there.  We show that large
positive right secants occur at every interior point, even at points
where $D$ is continuous.

The broader context is the study of functions whose digit expansions or
series representations produce regularity at one scale and oscillation
at a finer scale.  The Takagi function is a classical binary model for
this phenomenon \cite{AllaartKawamura2012}.  The present function differs
from that continuous model: it has a dense set of dyadic
discontinuities, and its coefficients arise from multiple zeta star
tails rather than a prescribed tent map series.  The issue is thus to
extract quantitative regularity from the arithmetic recursion itself.
The four results below describe the resulting local and global geometry
of $D$.

\subsection{Main results}

We always use the eventually-zero binary expansion at a dyadic point.
Let
\begin{equation}\label{eq:dyadic-set}
 \mathbb D=\{m2^{-n}:m\in\mathbb Z,\ n\in\mathbb Z_{\ge0}\},
 \qquad
 \mathcal C=(0,1/2)\setminus\mathbb D.
\end{equation}
For a word $w=w_1\cdots w_M$ with $w_j\in\{0,1\}$, put
\begin{equation}\label{eq:xw-def}
 x_w=\sum_{j=1}^Mw_j2^{-j},
 \qquad
 I_w=[x_w,x_w+2^{-M})\cap(0,1/2).
\end{equation}
An admissible word starts with zero and ends with one.  We call such a
word \emph{regular} if it contains at least five zeros. 
The first result identifies exactly where the derivative is
continuous and how its finite index values fit into the surrounding
infinite depth family.

\begin{theorem}\label{thm:continuity}
The function $D$ is finite and right-continuous on $(0,1/2)$, and
\begin{equation}\label{eq:continuity-set-main}
 \operatorname{Cont}(D)=\mathcal C.
\end{equation}
Here $\operatorname{Cont}(D)$ is the set of points where $D$ is continuous.
If $y=\beta^f(\mathbf k)$, where $\mathbf k=(k_1,\ldots,k_r)$ has weight
$M$, then the extended left limit exists and
\begin{equation}\label{eq:jump-main}
 \begin{split}
 J(\mathbf k)&:=D(y^-)-D(y)\\
 &=2^M\sum_{n_1\ge\cdots\ge n_r\ge3}
       \frac{n_r-2}{n_1^{k_1}\cdots n_r^{k_r}}
 \ \ge\ (2/3)^M>0.
 \end{split}
\end{equation}
The jump is infinite precisely when
$\mathbf k=(2,\{1\}^{r-1})$, equivalently when
$y=1/2-2^{-M}$ with $M\ge2$.  In particular, for
\begin{equation}\label{eq:dyadic-left-sequence-main}
 y_q=y-2^{-(M+q)},\qquad q\ge1,
\end{equation}
one has $D(y_q)\to D(y^-)$ in the extended real line. \end{theorem}

The positive series in \eqref{eq:jump-main} makes the one sided
instability arithmetically explicit.  In particular, a finite value of
$D$ may coexist with an infinite left limit; right continuity must not
be confused with local boundedness across such a point.  The next result
shows that irregularity is not confined to these exceptional limits or
to the dyadic discontinuities.

\begin{theorem}\label{thm:nowhere}
For every $y\in(0,1/2)$,
\begin{equation}\label{eq:dini-main}
 \limsup_{\substack{h\rightarrow0^+\\ y+h<1/2}}
 \frac{D(y+h)-D(y)}h=+\infty.
\end{equation}
Consequently, $D$ has no finite right derivative and no finite ordinary
derivative at any interior point.
\end{theorem}

Thus even on its continuity set, $D$ does not admit a finite local
linearization from the right.  This gives an arithmetic example of
nowhere finite differentiability with an explicitly known, dense
exceptional set for continuity.

To quantify the global oscillation, we next consider the graph.
For classical Weierstrass functions with horizontal scale $b^{-1}$ and
vertical scale $\lambda$, where $b\geq2$ is an integer and
$1/b<\lambda<1$, the expression
$2+\log\lambda/\log b$ gives the Hausdorff dimension in Shen's theorem
\cite{Shen2018}.  In our setting the corresponding scales are $1/2$ and
$2/3$, but the graph is discontinuous and the leading coefficient depends
on the binary prefix.  The analogy suggests the dimension to seek.

For the graph, write
\begin{equation}\label{eq:Gamma-intro}
 \Gamma_D=\big{\{}(y,D(y)):0<y<1/2\big{\}},
\end{equation}
and
\begin{equation}\label{eq:graph-cylinder-def}
 \Gamma_D(w)=\{(y,D(y)):y\in I_w\},
 \qquad
 \Gamma_{D,L}=\{(y,D(y)):0<y<2^{-L}\}.
\end{equation}
For a bounded set $E\subset\mathbb R^2$, let $N_\delta(E)$ denote the least
number of closed axis-parallel squares of side $\delta$ covering $E$, and define
\begin{equation}\label{eq:box-defs}
 \ldimB E=\liminf_{\delta\rightarrow0^+}
 \frac{\log N_\delta(E)}{-\log\delta},
 \qquad
 \udimB E=\limsup_{\delta\rightarrow0^+}
 \frac{\log N_\delta(E)}{-\log\delta}.
\end{equation}
For an arbitrary set, define
\begin{equation}\label{eq:modified-box-def}
 \udimMB E=
 \inf_{E\subset\bigcup_{i\ge1}E_i}
 \sup_i\udimB E_i,
 \qquad E_i\text{ bounded}.
\end{equation}
This is the packing dimension $\dimP E$ for Euclidean sets (Chapters $2.3$ and $3.5$, \cite{Falconer2014}).
 We use $\dimH$ for Hausdorff dimension.
Set
\begin{equation}\label{eq:constants-intro}
 \lambda=\frac23,
 \qquad h_*=-\log_2\lambda=\log_2\frac32,
 \qquad s_*=2-h_*=\log_2\frac83.
\end{equation}

\begin{theorem}\label{thm:dimensions}
One has
\begin{equation}\label{eq:global-main-dims}
 \dimH\Gamma_D=\dimP\Gamma_D=\udimMB\Gamma_D=s_*.
\end{equation}
More precisely:
\begin{enumerate}[label=\textnormal{(\roman*)}]
\item If $w$ is regular, then $\Gamma_D(w)$ is bounded and
\begin{equation}\label{eq:regular-main-dims}
 \dimH\Gamma_D(w)=\ldimB\Gamma_D(w)=\udimB\Gamma_D(w)=s_*.
\end{equation}
\item For every integer $L\ge1$, $\dimH\Gamma_{D,L}=s_*$.  If $L\ge5$, then
$\Gamma_{D,L}$ is bounded and
\begin{equation}\label{eq:zero-main-dims}
 \ldimB\Gamma_{D,L}=\udimB\Gamma_{D,L}=s_*.
\end{equation}
\end{enumerate}
\end{theorem}

The full graph is unbounded, so the global assertion uses the
countably stable dimensions.  Ordinary lower and upper box
dimensions are asserted on the bounded pieces specified in the theorem.
Throughout, the graph only contains the actual values of $D$; vertical
segments are not inserted at its jumps.

A graph dimension is a single global invariant and does not specify
where the strongest oscillations occur.  Pointwise H\"older exponents and
their Hausdorff spectrum address that finer question.  Jaffard's analysis
of Riemann's function \cite{Jaffard1996Riemann} and the study of Davenport
series by Durand--Jaffard \cite{DurandJaffard2013} illustrate how arithmetic
approximation and the location of singularities enter multifractal
analysis.  Here the relevant approximants are dyadic, and the lower
bounds for jumps in \eqref{eq:jump-main} provide a direct way to test
pointwise regularity.

For $y\in\mathcal C$, define its pointwise H\"older exponent by
\begin{equation}\label{eq:holder-exponent-main}
 \begin{split}
 h_D(y)=\sup\biggl\{h\ge0:\ &\exists\, C,r_0>0\text{ such that}\\
 &|D(z)-D(y)|\le C|z-y|^h\text{ for all }\\
 &z\in(0,1/2)\text{ with }0<|z-y|<r_0\biggr\}.
 \end{split}
\end{equation}
Only exponents below one will occur, so this agrees with the pointwise
polynomial-approximation convention used in multifractal analysis
\cite{JaffardMeyer1996,Jaffard1997I}.  Put
\begin{equation}\label{eq:tau2-main}
 \|t\|=\min_{m\in\mathbb Z}|t-m|,
 \quad
 \delta_n(y)=2^{-n}\|2^ny\|,
 \quad
 \tau_2(y)=\limsup_{n\to\infty}
 \frac{-\log_2\delta_n(y)}n.
\end{equation}
In particular, $1\le\tau_2(y)\le+\infty$.

\begin{theorem}\label{thm:holder-spectrum-main}
For every $y\in\mathcal C$,
\begin{equation}\label{eq:pointwise-holder-main}
 h_D(y)=\frac{h_*}{\tau_2(y)},
 \qquad \frac{h_*}{+\infty}=0.
\end{equation}
With $\dimH\varnothing=-\infty$, one has
\begin{equation}\label{eq:holder-spectrum-main}
 \dimH\{y\in\mathcal C:h_D(y)=h\}
 =\begin{cases}
 h/h_*,&0\le h\le h_*,\\
 -\infty,&h>h_*.
 \end{cases}
\end{equation}
Moreover, $h_D(y)=h_*$ for Lebesgue-almost every $y$.  Each value in
$[0,h_*]$ occurs densely.  The zero-exponent set is dense and uncountable,
but has Hausdorff dimension zero.
\end{theorem}

The formula \eqref{eq:pointwise-holder-main} separates a universal
arithmetic scale, $h_*$, from the quality of approximation of the point
by dyadic rationals.  Long constant blocks in a binary expansion can place
the point unusually close to such a dyadic point and reduce its H\"older exponent.
The spectrum therefore distinguishes typical local stability from
exceptional sensitivity, even though both occur densely.  The exponent
is a supremum: the theorem does not assert that the H\"older estimate is
attained at the critical exponent.  Taken together, the four results
show how a single boundary contribution to a star tail controls the
regularity scale.

Let $m=\LL|_{(0,1/2)}$ be the Lebesgue measure on
$(0,1/2)$, and let $\rho=D_*m$ be its pushforward under
$D=\Rp\Phi$.  Thus $\rho$ is the Borel measure on $\R$
defined by
\begin{equation}\label{eq:global-vertical-distribution}
 \rho(A)
 =\LL\bigl(\{y\in(0,1/2):D(y)\in A\}\bigr)
\end{equation}
for every Borel set $A\subset\R$. Inspired by Solomyak's main results \cite{Solomyak1995}, we propose the following conjecture.

\begin{conjecture}
\label{conj:global-vertical-density}
The measure $\rho$ is absolutely continuous with respect to
Lebesgue measure and has a square-integrable density.
More precisely, there exists a nonnegative function
$f\in L^1(\R)\cap L^2(\R)$ such that, for every Borel set
$A\subset\R$,
\begin{equation}\label{eq:global-vertical-density}
 \LL\bigl(\{y\in(0,1/2):D(y)\in A\}\bigr)
 =\int_A f(u)\,du.
\end{equation}
In particular,
\[
 \int_{\R}f(u)\,du=\frac{1}{2},
 \qquad
 \int_{\R}f(u)^2\,du<+\infty.
\]
\end{conjecture}

\subsection{Arithmetic mechanism, applications, and organization}

The common mechanism behind the results remains within finite index
multiple zeta-star arithmetic.  For a finite index $\mathbf k$ of weight
$K$, the contribution of the boundary $n_r=m$ to a normalized star tail is
\begin{equation}\label{eq:alpha-intro}
 \alpha_m(\mathbf k)=
 2^K\sum_{n_1\ge\cdots\ge n_r=m}
 \frac1{n_1^{k_1}\cdots n_r^{k_r}},\qquad m\ge3.
\end{equation}
Appending an entry $k$ gives the positive recursion
\begin{equation}\label{eq:terminal-intro}
 \alpha_m(\mathbf k,k)
 =\left(\frac2m\right)^k
 \sum_{n\ge m}\alpha_n(\mathbf k).
\end{equation}
These are boundary coefficients of familiar lower truncated star sums,
not additional special values.  A cutoff decomposition expresses each
fixed-cutoff tail through star values of prefixes and finite multiple
harmonic star sums.  This keeps both the evaluations and the estimates
connected to the standard finite index theory.

The decisive step is to separate $m=3$ from $m\ge4$.  The homogeneous
part of the first coefficient produces the vertical contraction
$\lambda=2/3$; higher coefficients supply a nonnegative forcing term.
The identity
\[
 \sum_{m=4}^n\frac2m\binom m4=\frac12\binom n4
\]
yields a weighted contraction $1/2$ for the remaining coefficients.
After fixing a regular initial prefix and then $N$ additional binary
digits, the leading coefficient is comparable to $\lambda^N$, while
the new remainder is $O(2^{-N})$, uniformly over continuations.
The difference between these two rates explains the constants
$h_*=\log_2(3/2)$ and $s_*=2-h_*$.  This is a
\emph{cylinderwise} decomposition: it neither gives a globally Lipschitz
remainder nor makes the leading coefficient independent of the prefix.

For the Hausdorff lower bound, the leading symbolic series has the law
of a fair Bernoulli convolution with parameter $2/3$.  Bernoulli
convolutions relate overlapping self-similar measures to arithmetic
properties of their contraction parameter; see
\cite{PeresSchlagSolomyak2000}.  Solomyak's almost-everywhere absolute
continuity theorem \cite{Solomyak1995} is an important part of that
theory, but an almost everywhere parameter statement alone does not
settle a prescribed parameter.  We instead use exponential separation
at $2/3$, Hochman's dimension theorem, and Feng--Hu exact dimensionality
\cite{Hochman2014,FengHu2009}.  A separate measure transfer argument then
passes from the one dimensional convolution to the arithmetic graph.
No absolute continuity or uniform density bound for that convolution is
assumed.  For the pointwise spectrum, the jump test reduces the problem
to dyadic approximation, and the mass transference principle of
Beresnevich--Velani \cite{BeresnevichVelani2006} supplies the Hausdorff
lower bounds for the exact approximation levels.

These estimates have concrete uses within the zeta star correspondence.
First, \eqref{eq:finite-D-intro} transfers identities for multiple zeta
values to derivative evaluations.  Section~\ref{sec:examples} gives, for
example,
\[
 D(2^{-k})=2^k\bigl(\zeta(k)-1\bigr)\quad(k\ge2),
 \qquad D(1/3)=6,
\]
as well as explicit jump formulas and infinite product evaluations for
periodic indices.  Second, the symbolic oscillation estimate gives
uniform $O((2/3)^N)$ control of the uncertainty left by $N$ additional
binary digits within a fixed regular cylinder.  Retaining the leading
symbolic contribution improves the remaining binary-tail error to
$O(2^{-N})$; see
\eqref{eq:conditional-decomposition}--\eqref{eq:symbolic-oscillation}.
These are quantitative bounds for finite prefix approximation of $D$.
The distinct error from truncating the summation variable $m$ is bounded
in Corollary~\ref{cor:cutoff-error}; numerical errors in the initial
coefficients and the finite-prefix value must still be controlled
separately.  Third, the H\"older formula identifies the
points at which Euclidean error control deteriorates because of
exceptionally close dyadic approximation.  It complements, rather than
replaces, the metric approximation results for star values themselves.

The same separation of a leading boundary coefficient from a faster
contracting tail suggests a possible approach to deformed or
restricted-index correspondences.  Such an extension would require
its own positive recursion, contraction estimates, and dimension theorem
for the leading symbolic measure.  We do not establish these conditions
for the rational or complex variations cited above.

The logical inputs are separated as follows.  The zeta-star correspondence
and the derivative formula are imported from \cite{HMO2025,LiTopology2023}.
Proposition~\ref{prop:derivative-input} records the latter with its precise
normalization and hypotheses.  The tail identities, binary recursion,
jumps, and large right secants are then proved directly.  The geometric
arguments additionally use the dimension and exact dimensionality
theorems just cited, and the mass transference principle.  

The paper is organized as follows.  Section~\ref{sec:truncated}
establishes the truncated star identities and boundary coefficients.
Section~\ref{sec:binary} proves the two contraction estimates and the
cylinderwise decomposition.  Section~\ref{sec:regularity} treats the
one sided limits and infinite upper right Dini derivatives.
Section~\ref{sec:dimensions} proves the graph dimension formulas, and
Section~\ref{sec:holder} determines the pointwise exponents and their
Hausdorff spectrum.  Explicit finite index and periodic index
evaluations, together with the scope of the method, are collected in
Section~\ref{sec:examples}.

\section{Truncated multiple zeta-star sums}\label{sec:truncated}

\subsection{Harmonic sums and star tails}

For a finite index $\mathbf k=(k_1,\ldots,k_r)$, not necessarily admissible,
let
\begin{equation}\label{eq:harmonic-star}
 H_N^\star(\mathbf k)
 =\sum_{N\ge n_1\ge\cdots\ge n_r\ge1}
 \frac1{n_1^{k_1}\cdots n_r^{k_r}}.
\end{equation}
The strict version is denoted by $H_N(\mathbf k)$.  Empty-index values are
one; for a nonempty index we set $H_0^\star(\mathbf k)=0$.
For an admissible index and $q\ge1$, define the lower-truncated star tail
\begin{equation}\label{eq:star-tail}
 Z_q^\star(\mathbf k):=\zeta_{\ge q}^\star(\mathbf k)
 =\sum_{n_1\ge\cdots\ge n_r\ge q}
 \frac1{n_1^{k_1}\cdots n_r^{k_r}},
 \qquad Z_q^\star(\varnothing)=1.
\end{equation}
The auxiliary symbol $Z_q^\star$ is used to keep later formulas short.
In particular, $Z_1^\star=\zeta^\star$.

Write $\mathbf l\lhd\mathbf k$ when $\mathbf l$ is obtained from
$\mathbf k$ by merging some adjacent entries.  Partitioning a weakly
ordered summation domain according to its equalities gives
\begin{equation}\label{eq:coarsening}
 \zeta^\star(\mathbf k)=\sum_{\mathbf l\,\lhd\,\mathbf k}\zeta(\mathbf l),
 \qquad
 H_N^\star(\mathbf k)=\sum_{\mathbf l\,\lhd\,\mathbf k}H_N(\mathbf l).
\end{equation}
The same identity holds with the lower cutoff $q$ in every summation.
Thus all the finite-index quantities used below can equally be written in
strict MZV notation.  The star notation is preferable here because
concatenation and the boundary recursions remain positive.

\begin{lemma}\label{lem:cutoff}
For an admissible index $\mathbf k=(k_1,\ldots,k_r)$ and $q\ge1$,
\begin{equation}\label{eq:cutoff-convolution}
 \zeta^\star(\mathbf k)
 =\sum_{j=0}^r
 Z_q^\star(\mathbf k^{[j]})
 H_{q-1}^\star(k_{j+1},\ldots,k_r).
\end{equation}
In particular,
\begin{equation}\label{eq:cutoff-two}
 Z_2^\star(\mathbf k)
 =\zeta^\star(\mathbf k)-\zeta^\star(\mathbf k^{[r-1]}).
\end{equation}
\end{lemma}

\begin{proof}
In each summand defining $\zeta^\star(\mathbf k)$ there is a unique $j$
for which $n_1,\ldots,n_j\ge q$ and $n_{j+1},\ldots,n_r<q$.
The two groups of variables are then independent, which gives
\eqref{eq:cutoff-convolution}.  For \eqref{eq:cutoff-two}, the terms
excluded by the condition $n_r\ge2$ are precisely those with $n_r=1$;
their sum is $\zeta^\star(\mathbf k^{[r-1]})$.
\end{proof}

For fixed $q$, recursion in \eqref{eq:cutoff-convolution} expresses
$Z_q^\star(\mathbf k)$ as a rational linear combination of star values of
prefixes of $\mathbf k$, since every finite harmonic sum appearing there
is rational.  This observation also identifies the arithmetic content of
the boundary coefficients below.

We will use the complete-homogeneous-sum identity
\begin{equation}\label{eq:all-one-generating}
 \sum_{j=0}^{\infty}
 H_{[q,m]}^\star(\{1\}^j)t^j
 =\prod_{n=q}^m\left(1-\frac tn\right)^{-1},
 \qquad |t|<q,
\end{equation}
where $H_{[q,m]}^\star$ means that every variable lies in $[q,m]$.
It follows directly by assigning a nonnegative multiplicity to each
integer $q,\ldots,m$.  More generally, replacing $1/n$ by $1/n^a$
gives the repeated-entry version.  This is the usual generating-function
interpretation of multiple harmonic star sums.

For completeness, the analogous multiplicity argument with an eventual
summation variable equal to one can be made explicit.  For a fixed
integer $m\ge1$, ones fill the unused positions, so
\[
 H_m^\star(\{1\}^q)
 =\sum_{\substack{c_2,\ldots,c_m\ge0\\c_2+\cdots+c_m\le q}}
   \prod_{n=2}^m n^{-c_n}
 \rightarrow\prod_{n=2}^m(1-1/n)^{-1}=m.
\]
The empty product at $m=1$ is one.  Applying this limit to the final
block of a star sum gives
\begin{equation}\label{eq:infinite-ones-identity}
 \begin{split}
 &\zeta^\star(k_1,\ldots,k_{r-1},k_r+1,\{1\}^{\infty})\\
 &\quad=
 \sum_{n_1\ge\cdots\ge n_r\ge1}
 \frac1{n_1^{k_1}\cdots n_r^{k_r+1}}
 \prod_{m=2}^{n_r}\left(1-\frac1m\right)^{-1}
 =\zeta^\star(\mathbf k).
 \end{split}
\end{equation}
Here the finite product is $n_r$ and all limiting interchanges follow from
monotone convergence.  Since the binary coordinate of the index on the
left is $\beta^f(\mathbf k)$, this proves \eqref{eq:finite-phi} from the
infinite-index correspondence; see also \cite{LiTopology2023,Li2025}.

\subsection{The arithmetic derivative formula}

Define
\begin{equation}\label{eq:normalized-tails}
 S(\mathbf k)=2^{\wt(\mathbf k)}Z_2^\star(\mathbf k),
 \qquad
 A(\mathbf k)=2^{\wt(\mathbf k)}Z_3^\star(\mathbf k),
 \qquad S(\varnothing)=1.
\end{equation}
Only the next proposition uses differentiability results for
zeta star correspondence.

\begin{proposition}\label{prop:derivative-input}
The right derivative $D(y)$ is finite for every $0<y<1/2$.
For a finite admissible index,
\begin{equation}\label{eq:finite-derivative}
 D\bigl(\beta^f(\mathbf k)\bigr)
 =S(\mathbf k)
 =2^{\wt(\mathbf k)}
 \bigl(\zeta^\star(\mathbf k)-\zeta^\star(\mathbf k^{[r-1]})\bigr).
\end{equation}
For a non-dyadic $y=\beta(\mathbf k)$,
\begin{equation}\label{eq:Li-intro}
 D(y)=\Phi'(y)=\lim_{r\to\infty}S(\mathbf k^{[r]}).
\end{equation}
The extension $D(0)=1$ agrees with $\partial_+\Phi(0)$.
\end{proposition}

\begin{proof}
Write $y=\sum_{j\ge1}b_j2^{-j}$ canonically, so $b_1=0$ and
$2y=\sum_{j\ge1}b_{j+1}2^{-j}$.  Since
$Z^\star(x)=\Phi(x/2)$, the derivative input is
$D(y)=2\partial_+Z^\star(2y)$.  Every canonical expansion of
$2y\in[0,1)$ has infinitely many zeros, including a terminating
expansion.  Thus Theorem~4.8 of \cite{HMO2025} applies to the right
derivative.  At a non-dyadic point there are also infinitely many ones
and at least two zeros, so the hypotheses of its Theorem~4.7 are
satisfied and the left derivative has the same finite value.
Theorem~4.10 treats the terminating case explicitly.  With the factor
of two from the change of variables included, these formulas give
\begin{equation}\label{eq:derivative-binary-input}
 D(y)=1+\sum_{d\ge1}b_{d+1}2^{d+1}
 \sum_{m_1\ge\cdots\ge m_d\ge3}
 \frac{\prod_{i=1}^{d-1}b_{i+1}^{m_i-m_{i+1}}}
      {m_1^2m_2\cdots m_d},
\end{equation}
with $0^0=1$ and an empty product equal to one.  This is a finite sum
in $d$ at a terminating expansion.  These theorems also give finiteness
and equality of the left and right derivatives at non-dyadic points.

The nonzero outer terms have $d+1=K_j$.  The zero digits force equalities
among adjacent $m_i$; grouping each resulting block gives exponents
$k_1,\ldots,k_j$.  The term is therefore
$2^{K_j}Z_3^\star(\mathbf k^{[j]})$.  On the other hand, partitioning
$Z_2^\star(\mathbf k)$ according to the last variable at least three gives
\[
 2^{K_r}Z_2^\star(\mathbf k)
 =1+\sum_{j=1}^r2^{K_j}Z_3^\star(\mathbf k^{[j]}).
\]
This identifies both the finite formula and the increasing sequence of
partial sums at a non-dyadic point. One can also consult Proposition $2.1$ in \cite{Li2026}. Equation~\eqref{eq:cutoff-two} gives the
star-value difference.  With every $b_j=0$, the same right-derivative
formula gives $D(0)=1$.
\end{proof}

\begin{lemma}\label{lem:prefix-expansion}
For every finite admissible index,
\begin{equation}\label{eq:prefix-expansion}
 S(\mathbf k)=1+\sum_{j=1}^r A(\mathbf k^{[j]}).
\end{equation}
Consequently, for a non-dyadic coordinate,
\begin{equation}\label{eq:infinite-positive-expansion}
 D\bigl(\beta(\mathbf k)\bigr)
 =1+\sum_{j=1}^{\infty}A(\mathbf k^{[j]}).
\end{equation}
\end{lemma}

\begin{proof}
In the sum $Z_2^\star(\mathbf k)$, specify the last variable that is at
least three.  If there are $j$ such variables, the remaining variables
are all two, so their factor is $2^{-(K_r-K_j)}$.  After multiplication
by $2^{K_r}$, this contribution is $A(\mathbf k^{[j]})$.
The case $j=0$ contributes one.  Taking the limit in
Proposition~\ref{prop:derivative-input} proves the infinite version. Lemma $5.5$ in \cite{Li2026} prove that for $\beta({\bf k})$ is non-dyadic, 
$D(\beta({\bf k}))<+\infty$. 
\end{proof}

\subsection{Boundary coefficients}

For $m\ge3$ and a nonempty admissible index, set
\begin{equation}\label{eq:alpha-definition}
 \begin{split}
 \alpha_m(\mathbf k)
 &=2^{K_r}m^{-k_r}Z_m^\star(\mathbf k^{[r-1]})\\
 &=2^{K_r}\bigl(Z_m^\star(\mathbf k)-Z_{m+1}^\star(\mathbf k)\bigr).
 \end{split}
\end{equation}
Thus
\begin{equation}\label{eq:alpha-mass}
 A(\mathbf k)=\sum_{m\ge3}\alpha_m(\mathbf k),
 \qquad
 \alpha_3(\mathbf k)\ge\lambda^{\wt(\mathbf k)}.
\end{equation}
The second inequality is the contribution with every variable equal to
three.  Directly from the last variable,
\begin{equation}\label{eq:alpha-recursion}
 \alpha_m(\mathbf k,k)
 =\left(\frac2m\right)^k\sum_{n\ge m}\alpha_n(\mathbf k),
 \qquad
 \alpha_m(k)=\left(\frac2m\right)^k.
\end{equation}
This is \eqref{eq:terminal-intro} in truncated-star notation.

\begin{lemma}\label{lem:alpha-decay}
Let $e(\mathbf k)=\wt(\mathbf k)-\dep(\mathbf k)$.  There are positive
constants $c_{\mathbf k},C_{\mathbf k}$ such that, for every $m\ge3$,
\begin{equation}\label{eq:alpha-decay}
 c_{\mathbf k}m^{-e(\mathbf k)-1}
 \le\alpha_m(\mathbf k)
 \le C_{\mathbf k}m^{-e(\mathbf k)-1}.
\end{equation}
In particular,
\begin{equation}\label{eq:jump-finiteness-criterion}
 \sum_{m\ge3}(m-2)\alpha_m(\mathbf k)<+\infty
 \quad\Leftrightarrow\quad e(\mathbf k)>1.
\end{equation}
\end{lemma}

\begin{proof}
For $\dep(\mathbf k)=1$, \eqref{eq:alpha-decay} is immediate.
If a positive sequence is comparable to $m^{-p}$ with $p>1$, then its
tail from $m$ is comparable to $m^{1-p}$, by the integral comparison for
$\sum_{n\ge m}n^{-p}$.  Apply this to \eqref{eq:alpha-recursion}.
Appending $k$ changes $e$ to $e+k-1$ and changes the decay exponent
from $e+1$ to $e+k$.  Induction proves the bounds.
An explicit admissible upper constant is obtained recursively from
\begin{equation}\label{eq:alpha-upper-constant}
 C_{(k)}=2^k,\qquad
 C_{(\mathbf k,k)}=2^kC_{\mathbf k}
       \left(\frac13+\frac1{e(\mathbf k)}\right).
\end{equation}
Indeed, for $m\ge3$ and $e=e(\mathbf k)\ge1$,
\[
 \sum_{n\ge m}n^{-e-1}
 \le m^{-e-1}+\int_m^\infty x^{-e-1}\,dx
 \le\left(\frac13+\frac1e\right)m^{-e}.
\]
This proves the upper bound with the displayed constants under the same
induction.  Finally, the last series in the statement is comparable,
after omitting finitely many terms, to $\sum_m m^{-e(\mathbf k)}$,
proving \eqref{eq:jump-finiteness-criterion}.
\end{proof}

The binary word of $\mathbf k$ is
\begin{equation}\label{eq:index-word}
 w(\mathbf k)=0^{k_1-1}1\,0^{k_2-1}1\cdots0^{k_r-1}1.
\end{equation}
Its length is the weight, its number of ones is the depth, and its
number of zeros is $e(\mathbf k)$.  We use the latter quantity only as
weight minus depth. It is not the MZV height, which counts entries
greater than one.

\section{Binary extensions and two contraction rates}\label{sec:binary}

\subsection{A binary form of the boundary recursion}

Write $c_m=2/m$ for $m\ge3$.  For a nonnegative sequence
$a=(a_m)_{m\ge3}$, define
\begin{equation}\label{eq:binary-operators}
 (T_0a)_m=c_ma_m,
 \qquad
 (T_1a)_m=c_m\sum_{n\ge m}a_n.
\end{equation}
These operators simply record whether a new binary digit continues an
entry of the index or inserts a comma.  In particular, they preserve
nonnegativity and $T_0a\le T_1a$ coordinatewise.

For a word $w$ of length $M\ge1$ starting with zero, let $\mathbf k$ be
the finite index determined by its ones.  If the last one is at $K_r$,
put $t=M-K_r$; if there are no ones, set $\mathbf k=\varnothing$ and
$K_r=0$.  Define the \emph{next boundary vector}
\begin{equation}\label{eq:next-vector}
 a_m(w)=2^{M+1}m^{-(t+1)}Z_m^\star(\mathbf k).
\end{equation}
This is $\alpha_m(\mathbf k,t+1)$ when $\mathbf k\ne\varnothing$, and
$\alpha_m(M+1)$ otherwise.  For the initial word $w=0$,
$a_m(0)=c_m^2$.

Let $d(w)=D(x_w)$, using $D(0)=1$.  Equations
\eqref{eq:prefix-expansion} and \eqref{eq:alpha-recursion} imply
\begin{equation}\label{eq:word-update}
 \begin{split}
 a(wb)&=T_ba(w),\\
 d(wb)&=d(w)+b\sum_{m\ge3}a_m(w),\qquad b\in\{0,1\}.
 \end{split}
\end{equation}
Indeed, $b=0$ leaves the finite coordinate unchanged and increases the
next possible index entry by one.  For $b=1$, the new boundary contribution
is $\sum_m a_m(w)$, and a subsequent entry of size one gives $T_1a(w)$.

For $\omega=(\omega_1,\omega_2,\ldots)\in\SigmaTwo$, set
\begin{equation}\label{eq:tail-function}
 \begin{split}
 a^{(0)}&=a,
 \qquad a^{(j)}=T_{\omega_j}a^{(j-1)},\\
 F_a(\omega)&=\sum_{j=0}^{\infty}
 \omega_{j+1}\sum_{m\ge3}a_m^{(j)},
 \qquad
 X(\omega)=\sum_{j=1}^{\infty}\omega_j2^{-j}.
 \end{split}
\end{equation}
At first $F_a$ is allowed to take the value $+\infty$.  For a general
nonnegative vector, first truncate its initial support to $3\le m\le Q$
and take finitely many terms of the displayed series, then let both
cutoffs increase.  Positivity makes these limits independent of their
order; a zero binary digit contributes no term.
Thus $F_a$ is linear in $a$ on nonnegative sequences, by monotone
convergence.  For a canonical
binary tail, \eqref{eq:word-update} and the derivative input give
\begin{equation}\label{eq:cylinder-representation}
 D\bigl(x_w+2^{-M}X(\omega)\bigr)
 =d(w)+F_{a(w)}(\omega),
\end{equation}
provided the coordinate is in $(0,1/2)$ and the expansion $w\omega$ is
canonical.  At an eventually-one expansion, the right side is its
symbolic value, which need not be the value of $D$ at the same coordinate.

\subsection{A harmonic star generating function}

Let $\mathbf e_m$ be the sequence with a one at $m$ and zeros elsewhere.

\begin{lemma}\label{lem:all-one-tail}
For every nonnegative sequence $a$,
\begin{equation}\label{eq:all-one-F}
 0\le F_a(\omega)\le F_a(\{1\}^\infty)
 =\sum_{m\ge3}\frac{m(m-1)}2a_m,
\end{equation}
with equality interpreted in $[0,+\infty]$.
\end{lemma}

\begin{proof}
The coordinatewise inequalities $T_b\le T_1$ show that the all-one tail
maximizes each nonnegative contribution in \eqref{eq:tail-function}.
For $a=\mathbf e_m$,
\[
 \sum_{n\ge3}(T_1^j\mathbf e_m)_n
 =2^jH_{[3,m]}^\star(\{1\}^j).
\]
Consequently, \eqref{eq:all-one-generating} at $t=2$ gives
\[
 F_{\mathbf e_m}(1^\infty)
 =\prod_{n=3}^m\left(1-\frac2n\right)^{-1}
 =\frac{m(m-1)}2.
\]
Linearity and monotone convergence complete the proof.
\end{proof}

This elementary use of $H^\star(\{1\}^j)$ will account both for the dyadic
jumps and for the uniform control of all possible binary tails.

\subsection{A weighted tail estimate}

Introduce the nonnegative weighted sum
\begin{equation}\label{eq:weighted-B}
 B(a)=\sum_{m\ge4}\binom m4 a_m.
\end{equation}
The summation starts at four, so the dominant coefficient $a_3$ is
separate.

\begin{lemma}\label{lem:two-rates}
Suppose $a_3<+\infty$ and $B(a)<+\infty$.  Then, for $b\in\{0,1\}$,
\begin{equation}\label{eq:B-contract}
 B(T_ba)\le\frac12B(a),
\end{equation}
and equality holds when $b=1$.  Along every word of length $N$,
\begin{equation}\label{eq:state-bounds}
 B(a^{(N)})\le2^{-N}B(a),
 \qquad
 \lambda^Na_3\le a_3^{(N)}
 \le\lambda^N\bigl(a_3+4B(a)\bigr).
\end{equation}
\end{lemma}

\begin{proof}
For $T_0$, use $c_m\le1/2$ for $m\ge4$.  For $T_1$, interchange the
nonnegative sums and use the binomial summation identity:
\begin{align*}
 B(T_1a)
 &=\sum_{n\ge4}a_n\sum_{m=4}^n\frac2m\binom m4\\
 &=\frac12\sum_{n\ge4}a_n\sum_{m=4}^n\binom{m-1}{3}
 =\frac12B(a).
\end{align*}
The first bound in \eqref{eq:state-bounds} follows by iteration.
For the coefficient at three, \eqref{eq:binary-operators} gives the
forced recursion
\begin{equation}\label{eq:forced-third-coefficient}
 a_3^{(j+1)}
 =\lambda a_3^{(j)}
  +\lambda\omega_{j+1}\sum_{m\ge4}a_m^{(j)}.
\end{equation}
Thus exact contraction by $\lambda$ holds for the homogeneous
$\mathbf e_3$ component, not for the full coefficient when the forcing
term is nonzero.  Since $\sum_{m\ge4}a_m^{(j)}\le B(a^{(j)})$, division by
$\lambda^{j+1}$ and summation give
\[
 a_3\le\lambda^{-N}a_3^{(N)}
 \le a_3+B(a)\sum_{j=0}^{N-1}(2\lambda)^{-j}
 \le a_3+4B(a).
\]
Here $(2\lambda)^{-1}=3/4$.
\end{proof}

\begin{lemma}\label{lem:zero-count}
Let $w$ start with zero, and let $e(w)$ be its number of zeros.  Then
$a_3(w)>0$, and
\begin{align}
 B(a(w))<+\infty
 &\quad\Leftrightarrow\quad e(w)\ge5,
 \label{eq:zero-weighted-threshold}\\
 \sup_{\omega\in\SigmaTwo}F_{a(w)}(\omega)<+\infty
 &\quad\Leftrightarrow\quad e(w)\ge3.
 \label{eq:zero-bounded-threshold}
\end{align}
In particular, the five-zero hypothesis is sufficient for all the
weighted estimates below, although three zeros suffice for bounded
symbolic tails.
\end{lemma}

\begin{proof}
If $w$ has $r$ ones and length $M$, the candidate index defining $a(w)$
in \eqref{eq:next-vector} has weight $M+1$ and depth $r+1$.
Its weight minus depth is $M-r=e(w)$.
Lemma~\ref{lem:alpha-decay} gives $a_m(w)\asymp_w m^{-e(w)-1}$.
Since $\binom m4\asymp m^4$, the weighted series is comparable, after
finitely many terms are omitted, to $\sum_m m^{3-e(w)}$.  This
converges exactly when $e(w)>4$.  By Lemma~\ref{lem:all-one-tail},
the supremum in \eqref{eq:zero-bounded-threshold} equals
$\sum_{m\ge3}m(m-1)a_m(w)/2$, which is comparable to
$\sum_m m^{1-e(w)}$ and converges exactly when $e(w)>2$.
Positivity at three is immediate.  The same argument includes a word
with no ones, whose candidate index is $(M+1)$.
\end{proof}

\begin{corollary}\label{cor:cutoff-error}
Let $a$ be nonnegative with $\sum_{m\ge3}m^2a_m<+\infty$, and let
$a^{\le Q}$ agree with $a$ for $3\le m\le Q$ and be zero otherwise,
where $Q\ge3$ is an integer.  Then, uniformly for $\omega\in\SigmaTwo$,
\begin{equation}\label{eq:uniform-cutoff-error}
 0\le F_a(\omega)-F_{a^{\le Q}}(\omega)
 \le\frac12\sum_{m>Q}m(m-1)a_m.
\end{equation}
For $a=a(w)$ with $e=e(w)\ge3$, choose $C_w$ from
\eqref{eq:alpha-upper-constant} for its candidate index.  Then
\begin{equation}\label{eq:polynomial-cutoff-error}
 \sup_{\omega\in\SigmaTwo}
 |F_{a(w)}(\omega)-F_{a(w)^{\le Q}}(\omega)|
 \le\frac{C_w}{2(e-2)}Q^{2-e}.
\end{equation}
\end{corollary}

\begin{proof}
The defining nonnegative series makes $F_a$ linear in $a$.
The second-moment hypothesis and Lemma~\ref{lem:all-one-tail} make
all values finite.  Apply that lemma to $a-a^{\le Q}$ to obtain
\eqref{eq:uniform-cutoff-error}.  For the second assertion use
$a_m(w)\le C_wm^{-e-1}$, $m(m-1)\le m^2$, and
\[
 \sum_{m>Q}m^{1-e}\le\int_Q^\infty x^{1-e}\,dx
 =\frac{Q^{2-e}}{e-2}.
\]
This proves \eqref{eq:polynomial-cutoff-error}.
\end{proof}

Both operators preserve vectors supported on $\{3,\ldots,Q\}$, so the
truncated model uses a finite-dimensional recursion.  The bounds just
proved concern deletion of the initial coefficients above $Q$. They do
not by themselves control numerical errors in the retained coefficients,
in $d(w)$, or in a separate truncation of the binary expansion.

Define
\[
U: \SigmaTwo   \rightarrow [0,3],
\]
\begin{equation}\label{eq:U-definition}
 U(\omega)=\sum_{j=0}^{\infty}\omega_{j+1}\lambda^j,
 \qquad 0\le U(\omega)\le3.
\end{equation}

\begin{proposition}\label{prop:decomposition}
Suppose $a_3<+\infty$ and $B(a)<+\infty$.  Then
\begin{equation}\label{eq:leading-remainder}
 F_a(\omega)=a_3U(\omega)+R_a(\omega),
 \qquad 0\le R_a(\omega)\le6B(a).
\end{equation}
For a finite word $v=v_1\cdots v_N$, put
$a(v)=T_{v_N}\cdots T_{v_1}a$ and
\[
 d_a(v)=\sum_{j=0}^{N-1}v_{j+1}
 \sum_{m\ge3}(T_{v_j}\cdots T_{v_1}a)_m,
\]
where the operator product at $j=0$ is the identity.  Then
\begin{equation}\label{eq:conditional-decomposition}
 F_a(v\omega)
 =d_a(v)+a_3(v)U(\omega)+R_{a(v)}(\omega),
\end{equation}
where, uniformly in $v$ and $\omega$,
\begin{equation}\label{eq:remainder-cylinder}
 0\le R_{a(v)}(\omega)\le6B(a)2^{-N},
 \qquad
 a_3\lambda^N\le a_3(v)
 \le\bigl(a_3+4B(a)\bigr)\lambda^N.
\end{equation}
If two sequences agree in their first $N$ digits, then
\begin{equation}\label{eq:symbolic-oscillation}
 |F_a(\omega)-F_a(\eta)|
 \le\bigl(3a_3+18B(a)\bigr)\lambda^N.
\end{equation}
In particular, $F_a$ is continuous on the compact product space
$\SigmaTwo$.
\end{proposition}

\begin{proof}
Both $T_0$ and $T_1$ act on $\mathbf e_3$ by multiplication by $\lambda$.
Thus $F_{a_3\mathbf e_3}=a_3U$.  The contribution of the other coordinates
is nonnegative and, by Lemma~\ref{lem:all-one-tail}, is at most
\[
 \sum_{m\ge4}\frac{m(m-1)}2a_m
 \le6\sum_{m\ge4}\binom m4a_m.
\]
The last inequality follows from
$\frac{m(m-1)}{2\binom m4}=\frac{12}{(m-2)(m-3)}\le6$.
This proves \eqref{eq:leading-remainder}.

Splitting the series after $N$ digits gives
$F_a(v\omega)=d_a(v)+F_{a(v)}(\omega)$.
Apply Lemma~\ref{lem:two-rates} to obtain
\eqref{eq:conditional-decomposition}--\eqref{eq:remainder-cylinder}.
The range of $F_{a(v)}$ has diameter at most
$3a_3(v)+6B(a(v))$, which is bounded by the right side of
\eqref{eq:symbolic-oscillation}.  Uniform decay of these cylinder
oscillations proves continuity.
\end{proof}

\begin{remark}\label{rem:two-remainders}
The remainder in \eqref{eq:leading-remainder} need not have a better
Euclidean modulus of continuity than the whole function.  The useful
statement is \eqref{eq:remainder-cylinder}: after recentering on a cylinder,
its \emph{new} remainder is $O(2^{-N})$.  Also, the leading coefficient
$a_3(v)$ depends on the prefix.  Neither independence of the remainder
nor a fixed global leading coefficient is needed in the arguments below.
\end{remark}

\begin{corollary}\label{cor:compact-graph}
For every word $w$ starting with zero and containing at least five zeros,
the symbolic graph
\begin{equation}\label{eq:compact-symbolic-graph}
 K_w=\bigl\{(x_w+2^{-|w|}X(\omega),
              d(w)+F_{a(w)}(\omega)):\omega\in\SigmaTwo\bigr\}
\end{equation}
is compact and bounded.  One has $\Gamma_D(w)\subset K_w$, and
$K_w\setminus\Gamma_D(w)$ consists of at most countably many points,
all over dyadic coordinates, including possibly endpoints.  The set $K_w$
does not insert vertical line segments at jumps.  Moreover,
\begin{equation}\label{eq:closure-graph}
 \overline{\Gamma_D(w)}=K_w.
\end{equation}
\end{corollary}

\begin{proof}
Compactness follows from Proposition~\ref{prop:decomposition} and the
continuity of $X$.  A non-dyadic coordinate has only one binary expansion,
and \eqref{eq:cylinder-representation} gives its value of $D$.
A dyadic coordinate has at most two expansions.  This proves the
countable-difference assertion.  Eventually-zero tails, omitting an
excluded endpoint if necessary, are dense in $\SigmaTwo$ and have the
correct values of $D$.  Their images are therefore dense in $K_w$.
\end{proof}

\section{One-sided limits and nowhere differentiability}\label{sec:regularity}

\subsection{The left jump at a finite-index coordinate}

\begin{proposition}\label{prop:dyadic-gap}
Let $\mathbf k=(k_1,\ldots,k_r)$ be admissible, let
$M=\wt(\mathbf k)$, and put $y=\beta^f(\mathbf k)$.  Then
\begin{equation}\label{eq:exact-left-gap}
 D(y^-)=D(y)+\sum_{m\ge3}(m-2)\alpha_m(\mathbf k)
\end{equation}
in the extended real line.  The limit is approached by the sequence in
\eqref{eq:dyadic-left-sequence-main}.
\end{proposition}

\begin{proof}
Write the terminating word of $y$ as $p1$, where $p$ has length $M-1$.
The candidate boundary vector immediately before this last digit is
$a_m=\alpha_m(\mathbf k)$.  Set $d_0=d(p)$, so
\[
 D(y)=d_0+\sum_{m\ge3}a_m.
\]
The alternative expansion of $y$ is $p0\,1^\infty$.  By
Lemma~\ref{lem:all-one-tail}, its symbolic value is
\begin{align*}
 d_0+F_{T_0a}(1^\infty)
 &=d_0+\sum_{m\ge3}\frac{m(m-1)}2\frac2m a_m\\
 &=d_0+\sum_{m\ge3}(m-1)a_m.
\end{align*}
Subtracting the finite number $D(y)$ gives
\eqref{eq:exact-left-gap}.

It remains to identify this symbolic value with the actual left limit.
The point $y_q$ has expansion $p0\,1^q0^\infty$, so its value is the
$q$-term partial sum of $d_0+F_{T_0a}(1^\infty)$.  These values increase
to the displayed symbolic value.  A general $z\rightarrow y^-$ eventually has
expansion $p0$ followed by arbitrarily many initial ones.  Nonnegativity
of the remaining contributions gives the same partial sums as lower
bounds, whereas the all-one tail is an upper bound by
Lemma~\ref{lem:all-one-tail}.  This proves convergence, including when
the upper bound is infinite.
\end{proof}

\begin{proof}[Proof of Theorem \ref{thm:continuity}]
Finiteness is part of Proposition \ref{prop:derivative-input} and Lemma \ref{lem:prefix-expansion}.
At a non-dyadic point, both zeros and ones occur infinitely often.
Choose an admissible prefix containing at least five zeros.  The point
lies in the interior of its cylinder.  As nearby coordinates approach
it, their canonical expansions share arbitrarily long prefixes with its
expansion.  Equation~\eqref{eq:symbolic-oscillation} therefore proves
continuity.

At a dyadic point $y$, prolong its terminating word by sufficiently many
zeros to obtain a word $v$ containing at least five zeros.  Then
$x_v=y$, and all sufficiently close points on its right lie in $I_v$.
Their tail sequences converge to $0^\infty$, where $F_{a(v)}=0$.
Proposition~\ref{prop:decomposition} proves right-continuity at $y$.

Proposition~\ref{prop:dyadic-gap} proves the left-limit assertion and,
by \eqref{eq:alpha-definition}, the series in \eqref{eq:jump-main}.
Its term at $m=3$ is at least $\lambda^M$ by
\eqref{eq:alpha-mass}; hence every dyadic point is a discontinuity.
By Lemma~\ref{lem:alpha-decay}, the jump diverges precisely when
$e(\mathbf k)=1$.  Since $k_1\ge2$, this is equivalent to
$\mathbf k=(2,\{1\}^{r-1})$.  Its word is $01^{M-1}$ and its coordinate
is $1/2-2^{-M}$.  This completes all the assertions.
\end{proof}

\begin{corollary}\label{cor:jump-size}
Let $w$ be a fixed word starting with zero and containing at least five
zeros.  There is a constant $C_w>0$ such that every dyadic point $q$ in
the interior of $I_w$, of reduced binary level $M$, satisfies
\begin{equation}\label{eq:jump-comparison}
 \lambda^M\le D(q^-)-D(q)\le C_w\lambda^M.
\end{equation}
\end{corollary}

\begin{proof}
The lower bound is already proved.  The candidate vector $a$ just before
the last one of $q$ is obtained from $a(w)$ after $M-|w|-1$ updates.
Since $(m-2)\le2\binom m4$ for $m\ge4$, its jump is bounded by
$a_3+2B(a)$.  Lemma~\ref{lem:two-rates} gives the upper bound, with a
constant depending only on $w$.
\end{proof}

\begin{remark}
The jumps are downward: $D(q^-) > D(q)$.  There is no conflict with the
intermediate-value property of ordinary derivatives, since $D$ is the
\emph{right} derivative of $\Phi$ on the whole interval, and $\Phi$ need
not be differentiable at the dyadic coordinates.  At the exceptional
dyadics the left limit is infinite, so the global graph is unbounded;
this is why ordinary box dimensions are stated only for bounded pieces.
\end{remark}

\subsection{Nowhere differentiability}

\begin{lemma}\label{lem:bit-flip}
Suppose the $N$th digit of the canonical binary expansion of $y\in(0,1/2)$
is zero, and $y+2^{-N}<1/2$.  Then
\begin{equation}\label{eq:bit-flip}
 D(y+2^{-N})-D(y)\ge\lambda^N.
\end{equation}
\end{lemma}

\begin{proof}
There is no carry in adding $2^{-N}$: only the $N$th digit changes.
The contribution newly inserted in \eqref{eq:word-update} is the total
mass of the candidate vector immediately before that digit.  Starting
from $a_m(0)=c_m^2$, each binary update multiplies its coefficient at
three by at least $\lambda$.  Thus this newly inserted mass is at least
$\lambda^N$.  Moreover, $T_1a\ge T_0a$, so every subsequent coefficient
for the modified expansion is at least the corresponding old
coefficient.  All subsequent contributions are nonnegative.  The finite
prefix inequalities therefore pass to the convergent derivative series
and prove \eqref{eq:bit-flip}.
\end{proof}

In index language, changing a zero to one splits an existing gap between successive partial
weights; if the changed digit lies after the last one of a terminating expansion, it appends a new
entry. The operator proof above covers both cases.

\begin{proof}[Proof of Theorem~\ref{thm:nowhere}]
Every canonical binary expansion in $(0,1/2)$ has infinitely many zeros,
including the terminating expansions.  Choose their positions $N\to+\infty$.
For all sufficiently large $N$, $y+2^{-N}<1/2$, and
Lemma~\ref{lem:bit-flip} gives
\[
 \frac{D(y+2^{-N})-D(y)}{2^{-N}}
 \ge(2\lambda)^N=(4/3)^N\rightarrow+\infty.
\]
This proves \eqref{eq:dini-main}.  A finite right derivative would force
all these quotients to have the same finite limit, which is impossible.
\end{proof}

\section{Dimensions of the graph}\label{sec:dimensions}

\subsection{The Bernoulli convolution}

Let $\mathbb P$ be the fair Bernoulli product measure on $\SigmaTwo$, and
let $\sigma$ be the left shift.  Define
\begin{equation}\label{eq:bernoulli-law}
 \nu=U_*\mathbb P.
\end{equation}
Thus $\nu$ is the self-similar probability measure for
\[
 u\mapsto\lambda u,
 \qquad u\mapsto1+\lambda u,
\]
with weights $1/2,1/2$.  This is an affine normalization of the usual fair
Bernoulli convolution with parameter $2/3$.

\begin{lemma}\label{lem:bernoulli-dimension}
Distinct level $n$ translations of this system have distance at least
$3^{-(n-1)}$.  The measure $\nu$ is exact dimensional of dimension one:
\begin{equation}\label{eq:nu-exact}
 \lim_{r\rightarrow\,0^+}
 \frac{\log\nu([u-r,u+r])}{\log r}=1
 \quad\text{for }\nu\text{-almost every }u.
\end{equation}
\end{lemma}

\begin{proof}
A difference of two level $n$ translations is
\[
 \sum_{j=0}^{n-1}d_j(2/3)^j,
 \qquad d_j\in\{-1,0,1\},
\]
with some $d_j\ne0$.  Multiplication by $3^{n-1}$ gives an integer.
This integer is nonzero: if $j_0$ is the first nonzero position, divide
it by $2^{j_0}$; its first nonzero term is odd and all the remaining
terms are even.  Its absolute value is therefore at least one.  This
proves the stated exponential separation.

Hochman's dimension theorem (Theorem $1.1$ in \cite{Hochman2014}) says that a
self-similar measure on the line can have dimension smaller than the
minimum of one and its similarity dimension only if cylinder
translations concentrate super-exponentially.  The established separation
excludes this case.  Here the similarity dimension is
$\log2/\log(3/2)>1$, so $\dim\nu=1$.
Feng--Hu exact dimensionality for self-similar measures (Theorem $2.8$  in  \cite{FengHu2009})
then gives \eqref{eq:nu-exact}.
\end{proof}

The following useful consequence of exact dimensionality does not assert
any uniform estimate at all points of the support.  For each
$0<\eps<1$, there exist a Borel set $E_\eps$ of positive $\nu$-measure
and constants $C_\eps,r_\eps>0$ such that
\begin{equation}\label{eq:good-small-balls}
 \nu([u-r,u+r])\le C_\eps r^{1-\eps}
 \quad(u\in E_\eps,\ 0<r<r_\eps).
\end{equation}
For an explicit measurable uniformization, define, for $L\ge1$,
\[
 E_{\eps,L}=\bigcap_{k\ge L}
 \left\{u\in[0,3]:
 \nu([u-2^{-k},u+2^{-k}])\le2^{-k(1-\eps)}\right\}.
\]
The ball-mass functions are Borel, so these sets are Borel.  Exact
dimensionality gives $\nu(\bigcup_LE_{\eps,L})=1$; choose $L$ with
$\nu(E_{\eps,L})>0$.  For $0<r<2^{-L}$, choose $k\ge L$ with
$2^{-k-1}<r\le2^{-k}$.  Monotonicity of ball mass then gives
\[
 \nu([u-r,u+r])\le2^{-k(1-\eps)}
 \le2^{1-\eps}r^{1-\eps}\qquad(u\in E_{\eps,L}).
\]
Thus one may take $E_\eps=E_{\eps,L}$,
$C_\eps=2^{1-\eps}$ and $r_\eps=2^{-L}$.  Only this
positive-measure uniformization is used below.

\subsection{A graph measure lemma}

For a nonnegative vector $a$ with $0<a_3<+\infty$ and $B(a)<+\infty$, put
\begin{equation}\label{eq:model-graph}
 K_a=\{(X(\omega),F_a(\omega)):\omega\in\SigmaTwo\},
 \qquad
 \mu_a=(X,F_a)_*\mathbb P.
\end{equation}
Here $\Btwo(z,r)$ denotes the Euclidean ball of radius $r$ about $z$.
The set $K_a$ is compact by Proposition~\ref{prop:decomposition}.
The measure $\mu_a$ gives zero mass to the points over dyadic horizontal
coordinates, because $X_*\mathbb P$ is Lebesgue measure on $[0,1]$.

\begin{proposition}\label{prop:graph-transfer}
For every such $a$,
\begin{equation}\label{eq:lower-local-mu}
 \liminf_{r\rightarrow0^+}
 \frac{\log\mu_a(\Btwo(z,r))}{\log r}\ge s_*
 \quad\text{for }\mu_a\text{-almost every }z.
\end{equation}
Consequently, $\dimH K_a\ge s_*$.
\end{proposition}

\begin{proof}
 Fix $0<\eps<1$.  Choose $E_\eps$ as in
\eqref{eq:good-small-balls}.  Recall that $\sigma$ is the left shift operator
\[
\sigma: \{0,1\}^{\mathbb{N}}\rightarrow \{0,1\}^{\mathbb{N}},\]
\[
\sigma(w_1,w_2,\cdots, w_r,\cdots)=(w_2,w_3,\cdots, w_r,\cdots),\quad w_i \in\{0,1\}, i\in\mathbb{Z}^+.
\]
 For $\mathbb P$-almost every $\omega$, the
set of integers
\[
 \mathcal G(\omega)=
 \{n\ge1:U(\sigma^n\omega)\in E_\eps\}
\]
has asymptotic density $\nu(E_\eps)>0$, by the ergodic theorem for the
Bernoulli shift.  If its increasing enumeration is $n_1<n_2<\cdots$,
then
\begin{equation}\label{eq:good-time-ratio}
 \frac{n_{j+1}}{n_j}\rightarrow1.
\end{equation}
For example, the density assertion implies
$j/n_j\to\nu(E_\eps)$, which proves \eqref{eq:good-time-ratio}.

Put
\begin{equation}\label{eq:radii}
 r_n=2^{-(1+\varepsilon)n}.
\end{equation}
The level-$n$ dyadic boundaries in $[0,1]$ are the $2^n+1$ points
$k2^{-n}$, $0\le k\le2^n$. Since $X_*\mathbb{P}$ is Lebesgue measure, the union bound
gives
\begin{equation}
 \mathbb{P}\{\mathrm{dist}(X(\omega),2^{-n}\mathbb{Z})\le r_n\} \le2(2^n+1)r_n \le4\,2^{-\varepsilon n}.
\end{equation}
The last expression is summable in $n$. The first Borel--Cantelli lemma
therefore implies that, for $\mathbb{P}$-almost every $\omega$, eventually
\begin{equation}\label{eq:boundary}
 \mathrm{dist}(X(\omega),2^{-n}\mathbb{Z})>r_n.
\end{equation}
No independence of these events is required.
Fix a sequence satisfying both this property and
\eqref{eq:good-time-ratio}, and write
$z=(X(\omega),F_a(\omega))$.

 If a coded point lies in
$\Btwo(z,r_n)$, its first $n$ digits agree with
those of $\omega$, apart from dyadic codes of measure zero.  Denote the
common prefix by $v$ and the remaining tail of that point by $\eta$.
By \eqref{eq:conditional-decomposition}, its vertical difference from
$z$ is
\[
 a_3(v)\bigl(U(\eta)-U(\sigma^n\omega)\bigr)
 +R_{a(v)}(\eta)-R_{a(v)}(\sigma^n\omega).
\]
Thus 
\[
\big{|}a_3(v)\bigl(U(\eta)-U(\sigma^n\omega)\bigr)
 +R_{a(v)}(\eta)-R_{a(v)}(\sigma^n\omega) \Big{|} \leq r_n\leq \frac{1}{2^n}.
\]
Since $a_3(v)\ge a_3\lambda^n$ and
$0\le R_{a(v)}\le6B(a)2^{-n}$, membership in the ball implies
\begin{equation}\label{eq:tail-small-ball}
 |U(\eta)-U(\sigma^n\omega)|
 \le C_a(3/4)^n,
 \qquad C_a=\frac{1+6B(a)}{a_3}.
\end{equation}
Here $r_n\le2^{-n}$ was used.  Conditional on the prefix, the tail still
has law $\mathbb P$.  Therefore, for all sufficiently large
$n\in\mathcal G(\omega)$,
\begin{align}
 \mu_a(\Btwo(z,r_n))
 &\le 2^{-n}
 \nu\bigl([U(\sigma^n\omega)-C_a(3/4)^n,
            U(\sigma^n\omega)+C_a(3/4)^n]\bigr)\notag\\
 &\le C'_{a,\eps}
 2^{-n}(3/4)^{n(1-\eps)}.\label{eq:graph-ball-bound}
\end{align}
Notice that no independence between the remainder and $U$ has been
used; the deterministic remainder bound is enough.

Set
$t_\eps=1+(1-\eps)\log_2(4/3)$.  Estimate
\eqref{eq:graph-ball-bound} gives the lower local-dimension bound
$t_\eps/(1+\eps)$ along the good radii.  It gives the same bound at all
radii.  Indeed, if $r_{n_{j+1}}<r\le r_{n_j}$, monotonicity of ball mass
gives
\[
 \frac{-\log_2\mu_a(\Btwo(z,r))}{-\log_2r}
 \ge\frac{t_\eps n_j-O(1)}{(1+\eps)n_{j+1}},
\]
and \eqref{eq:good-time-ratio} applies.  Intersecting the full-measure
sets for $\eps=1/k$, $k=2,3,\ldots$, proves
\eqref{eq:lower-local-mu}, because
$1+\log_2(4/3)=s_*$.

Fix $0<s<s_*$. By \eqref{eq:lower-local-mu}, for $\mu_a$-almost every $z$ there is
$r_z>0$ such that
\begin{equation}\label{eq:pointwiseFrostman}
 \mu_a(B_{\mathbb{R}^2}(z,r))\le r^s\qquad(0<r<r_z).
\end{equation}
Indeed, the defining logarithmic quotient is eventually greater than $s$;
multiplying by $\log r<0$ reverses the inequality and gives
\eqref{eq:pointwiseFrostman}.

We also need a common upper bound on the allowed radii. For integers $k\ge1$,
let
\[
 A_k=\{z\in K_a:\mu_a(B_{\mathbb{R}^2}(z,r))\le r^s
                  \text{ for every }0<r<1/k\}.
\]
These are Borel sets: it is enough to impose the inequalities at rational
radii, and the ball-mass functions of the center are Borel. Their union has
full $\mu_a$-measure, so choose $k$ with $\mu_a(A_k)>0$ and put $A=A_k$.

Take any countable cover $A\subset\bigcup_i E_i$ with
$d_i:=\mathrm{diam} E_i<1/(2k)$. Whenever $E_i$ meets $A$ and $d_i>0$, choose
$z_i\in E_i\cap A$. Then
\[
 E_i\subset B_{\mathbb{R}^2} (z_i,2d_i),\qquad
 \mu_a(B_{\mathbb{R}^2}  (z_i,2d_i))\le(2d_i)^s.
\]
A set of diameter zero that meets $A$ is a singleton of zero $\mu_a$-mass,
by \eqref{eq:pointwiseFrostman}, so those sets can be ignored in the following
mass estimate. Countable subadditivity now gives
\[
 0<\mu_a(A)\le\sum_i\mu_a( B_{\mathbb{R}^2}(z_i,2d_i))
 \le2^s\sum_i d_i^s,
\]
where the sum over balls is taken over the positive-diameter sets that meet
$A$. Thus every such cover has total $s$-power of diameters at least
$2^{-s}\mu_a(A)>0$. Taking the infimum over covers and letting their maximal
diameter tend to zero gives
\[
 \mathcal{H}^s(A)>0,
 \qquad\text{hence}\qquad \mathrm{dim}_{H} K_a\ge\mathrm{dim}_{H} A\ge s.
\]
Finally let $s\uparrow s_*$. This proves $\mathrm{dim}_H K_a\ge s_*$ and completes the
proof.

\end{proof}

\subsection{Upper coverings and the full graph}

\begin{proposition}\label{prop:model-dimension}
If $0<a_3<+\infty$ and $B(a)<+\infty$, then
\begin{equation}\label{eq:model-dimensions}
 \dimH K_a=\ldimB K_a=\udimB K_a=s_*.
\end{equation}
\end{proposition}

\begin{proof}
There are $2^N$ binary cylinders of length $N$.  On each one, the
horizontal range has length $2^{-N}$ and, by
\eqref{eq:symbolic-oscillation}, the vertical range has length at most
$C_a'\lambda^N$.  Consequently,
\begin{equation}\label{eq:graph-cover}
 N_{2^{-N}}(K_a)
 \le C_a''2^N\bigl(1+2^N\lambda^N\bigr)
 \le C_a'''(8/3)^N.
\end{equation}
Changing to arbitrary intermediate scales changes only a constant in
this covering bound.  Hence $\udimB K_a\le s_*$.
Proposition~\ref{prop:graph-transfer} and
$\dimH\le\ldimB\le\udimB$ prove equality throughout.
\end{proof}

\begin{proof}[Proof of Theorem~\ref{thm:dimensions}]
Let $w$ start with zero and contain at least five zeros.
Lemma~\ref{lem:zero-count} applies.  The affine map
\[
 (x,u)\longmapsto(x_w+2^{-|w|}x,d(w)+u)
\]
is bi-Lipschitz and maps $K_{a(w)}$ onto $K_w$.
Proposition~\ref{prop:model-dimension} therefore gives all three
Hausdorff and box dimensions of $K_w$ as $s_*$.  By
Corollary~\ref{cor:compact-graph}, $\Gamma_D(w)$ has the same closure
and differs from $K_w$ by at most countably many points.  Closure does
not change box dimensions.  Removing countably many points does not
change the positive Hausdorff dimension here.  This proves part (i).

Every non-dyadic coordinate belongs to an admissible cylinder with at
least five zeros: both digits occur infinitely often, so one can choose
a long enough prefix ending with one.  Thus $\Gamma_D$ is covered by
countably many such bounded graph pieces and by the countable graph
over the dyadic coordinates.  Their upper box bounds and
\eqref{eq:modified-box-def} give $\udimMB\Gamma_D\le s_*$.
On the other hand, any one regular cylinder gives
$\dimH\Gamma_D\ge s_*$.  Together with
$\dimH\le\dimP=\udimMB$, this proves \eqref{eq:global-main-dims}.

For $L\ge5$, apply the same argument to the word $0^L$.
Its next vector is $a_m=(2/m)^{L+1}$ and satisfies $B(a)<\infty$.
Its cylinder, with zero removed, is exactly $(0,2^{-L})$; hence the bounded-piece assertion in part (ii) follows.  Finally, for any
$L\ge1$, the interval $(0,2^{-L})$ contains a regular cylinder.  The
global upper bound and this cylinder's lower bound give
$\dimH\Gamma_{D,L}=s_*$.
\end{proof}

\section{Pointwise H\"older regularity and the multifractal spectrum}\label{sec:holder}

\subsection{Comparison with distance to dyadic grids}

\begin{proposition}\label{prop:pointwise-exponent}
For every non-dyadic $y\in(0,1/2)$,
\[
 h_D(y)=h_*/\tau_2(y).
\]
\end{proposition}

\begin{proof}
Choose a regular prefix whose cylinder contains $y$ in its interior.
All sufficiently close points $z$ lie in this fixed cylinder.
If the first different binary digit of $y$ and $z$ has position $j$,
\eqref{eq:symbolic-oscillation}, with the initial prefix absorbed into
the constant, gives
\begin{equation}\label{eq:holder-symbolic-upper}
 |D(z)-D(y)|\le C_y\lambda^j.
\end{equation}
A dyadic rational point of level $j$ separates $y$ and $z$, so
\begin{equation}\label{eq:separating-grid}
 \delta_j(y)\le |z-y|.
\end{equation}

Recall that \[
 \|t\|=\min_{m\in\mathbb Z}|t-m|,
 \quad
 \delta_n(y)=2^{-n}\|2^ny\|,
 \quad
 \tau_2(y)=\limsup_{n\to\infty}
 \frac{-\log_2\delta_n(y)}n.
\]
Suppose first that $\tau=\tau_2(y)<+\infty$.  For each $\eps>0$, the
definition of the limsup gives
$\delta_j(y)\ge2^{-(\tau+\eps)j}$ for all sufficiently large $j$.
The first differing position tends to infinity as $z\to y$, because
$y$ is non-dyadic.  Thus \eqref{eq:holder-symbolic-upper} and
\eqref{eq:separating-grid} give
\[
 |D(z)-D(y)|
 \le C_y|z-y|^{h_*/(\tau+\eps)}.
\]
Letting $\eps\rightarrow0^+$ proves $h_D(y)\ge h_*/\tau$.
When $\tau=+\infty$, the lower bound $h_D(y)\ge0$ follows from local
boundedness, already supplied by the fixed regular cylinder.

For the reverse inequality, suppose an exponent $h>0$ satisfies the
estimate in \eqref{eq:holder-exponent-main}.  Let $q_n$ be a nearest
point of $2^{-n}\mathbb Z$ to $y$, so $|q_n-y|=\delta_n(y)$.
For all sufficiently large $n$, $q_n$ is an interior point of the same
regular cylinder. Its left limit is finite there by
Corollary~\ref{cor:jump-size}.  If its reduced binary level is
$M_n\le n$, then
Theorem~\ref{thm:continuity} gives
\[
 D(q_n^-)-D(q_n)\ge\lambda^{M_n}\ge\lambda^n.
\]
Apply the assumed H\"older estimate both at $q_n$ and at points tending
to $q_n$ from the left.  The actual left-limit assertion is important
here.  The triangle inequality yields
\begin{equation}\label{eq:holder-jump-test}
 \lambda^n\le D(q_n^-)-D(q_n)
 \le2C\delta_n(y)^h.
\end{equation}
Taking base-two logarithms gives explicitly
\[
 h\,\frac{-\log_2\delta_n(y)}n
 \le h_*+\frac{\log_2(2C)}n.
\]
The limsup in \eqref{eq:tau2-main} therefore yields
$h\tau_2(y)\le h_*$.  If $\tau_2(y)=+\infty$, this excludes every
positive $h$.  Otherwise, taking the supremum over all admissible $h$
gives the required upper bound.
\end{proof}

The proof shows why closeness to \emph{either} side of a dyadic boundary
matters.  The cylinder estimate gives the upper oscillation scale
$\lambda^j$, while the jump at the intervening boundary prevents a
better estimate at a point that approaches such boundaries too rapidly.
The argument determines the supremal H\"older exponent; it does not claim
that the defining estimate must hold at the critical exponent itself.

\subsection{The exact spectrum of the dyadic approximation exponent}

For $t\ge1$, let
\begin{equation}\label{eq:tau-level-set}
 E_t=\{y\in\mathcal C:\tau_2(y)=t\},
 \qquad
 E_\infty=\{y\in\mathcal C:\tau_2(y)=+\infty\}.
\end{equation}
These are Borel sets: each $\delta_n$ is continuous and strictly positive
on $\mathcal C$, and $\tau_2$ is the limsup of Borel functions.

\begin{lemma}\label{lem:dyadic-spectrum}
In every nonempty open interval $I\subset(0,1/2)$,
\begin{equation}\label{eq:tau-spectrum}
 \dimH(E_t\cap I)=\frac1t\quad(1\le t<\infty),
 \qquad
 \dimH(E_\infty\cap I)=0.
\end{equation}
The set $E_1$ has full Lebesgue measure.  The set $E_\infty$ is
uncountable in every such interval.
\end{lemma}

\begin{proof}

{\bf Step 1.} For $t>1$, define the limsup set
\begin{equation}\label{eq:Wt}
 W_t=\{y\in(0,1/2):\delta_n(y)<2^{-tn}
                  \text{ for infinitely many }n\}.
\end{equation}
At level $n$, it is covered by $O(2^n)$ intervals of radius $2^{-tn}$.
For every $s>1/t$,
\[
 \sum_{n\ge1}2^n(2\cdot2^{-tn})^s<+\infty.
\]
So we have $\dimH W_t\le1/t$.
Since $\{y\in\mathcal C:\tau_2(y)\ge t\}$ is contained in every
$W_u$ with $1<u<t$, we also have
\begin{equation}\label{eq:tau-superlevel-upper}
 \dimH\{y\in\mathcal C:\tau_2(y)\ge t\}\le1/t
 \quad(t>1).
\end{equation}
{\bf Step 2.}
Fix $1<t<+\infty$ and set $s=1/t\in(0,1)$. We will prove the stronger
statement
\begin{equation}\label{eq:criticalW}
 \HH^s(W_t\cap I)=+\infty.
\end{equation}
The external result used for this step is the following specialization of Theorem $2$ in \cite{BeresnevichVelani2006}.

\begin{mtp}
Let $0<s<1$, and let $B_i=[x_i-r_i,x_i+r_i]$ be a sequence of balls in
$\R$ with $r_i>0$ and $r_i\to0$. Set
\[
 B_i^{(s)}=[x_i-r_i^s,x_i+r_i^s].
\]
Suppose that, for every nondegenerate bounded interval $J$,
\[
 \LL\!\left(J\cap\limsup_{i\to+\infty}B_i^{(s)}\right)=\LL(J).
\]
Then, for every such $J$,
\[
 \HH^s\!\left(J\cap\limsup_{i\to+\infty}B_i\right)
 =\HH^s(J)=+\infty.
\]
\end{mtp}
This is the one-dimensional case with dimension function $f(r)=r^s$.
The function $f$ is increasing, continuous, and tends to zero at zero.
Moreover, $f(r)/r=r^{s-1}$ is monotone, as required by the general theorem.

Here is an explicit sequence of balls to which it applies. At level
$n\geq1$, use all centers $m2^{-n}$ with $m\in\Z$ and $|m|\leq n2^n$,
and define
\begin{equation}\label{eq:balls}
 B_{n,m}=[m2^{-n}-r_n,m2^{-n}+r_n],
 \qquad r_n=2^{-tn-1}.
\end{equation}
The centers at that level are exactly the level $n$ dyadic rational points
in $[-n,n]$. Enumerate first all balls at level $1$, then all at level
$2$, and so on, in any order within each level. Every level is finite,
so the radii of the resulting sequence tend to zero. Repeated centers
at different levels cause no difficulty.

The enlarged radius is
\begin{equation}\label{eq:enlarged}
 r_n^s=(2^{-tn-1})^{1/t}=2^{-n-1/t}>2^{-n-1}.
\end{equation}
For any fixed $x\in\R$, choose $n>|x|+1$. A nearest level $n$ dyadic
point $q_n$ satisfies
\[
 |x-q_n|\leq2^{-n-1}<r_n^s,
 \qquad |q_n|\leq|x|+2^{-n-1}<n.
\]
Thus $q_n$ is among the centers used at level $n$, and $x$ belongs to an
enlarged ball at that level. This holds at every sufficiently large
level, so
\[
 \limsup_{i\to\infty}B_i^{(s)}=\R.
\]
In particular, the full-Lebesgue-measure hypothesis of the mass
transference principle holds in every interval, not only in $I$.

Let $$L_t=\limsup_{i\to+\infty}B_i$$ for the enumerated original balls.
Membership in infinitely many balls is equivalent to membership at
infinitely many distinct levels, since each level is finite. If
$y\in L_t\cap I$, then at each of those levels
\[
 \delta_n(y)\leq r_n=\tfrac12\,2^{-tn}<2^{-tn}.
\]
Consequently,
\begin{equation}\label{eq:LtW}
 L_t\cap I\subseteq W_t\cap I.
\end{equation}
The factor $1/2$ in~\eqref{eq:balls} ensures the strict inequality used
in $W_t$, while~\eqref{eq:enlarged} ensures that the enlarged balls still
cover the dyadic gaps.

Choose a nondegenerate closed interval $J\subset I$. All hypotheses of
the mass transference principle have now been checked, so
\[
 \HH^s(L_t\cap J)=\HH^s(J)=+\infty.
\]
  
For clarity, the last equality is an elementary fact about intervals:
if $J$ has length $\ell>0$ and is covered by sets of diameters
$d_i\leq\rho$, then $\sum_i d_i\geq\ell$, and hence
\[
 \sum_i d_i^s\geq\rho^{s-1}\sum_i d_i
 \geq\ell\rho^{s-1}\rightarrow\infty
 \quad\text{as }\rho\rightarrow0^+.
\]
Here $s<1$. Combining this conclusion with~\eqref{eq:LtW} proves
\eqref{eq:criticalW}.\\
{\bf Step 3.}
For $v>1$, denote by
\[
 F_v=\{y\in\CC:\tau_2(y)\geq v\}.
\]
For $1<t<+\infty$ and $s=1/t$, define
\[
 G_t=\{y\in\CC:\tau_2(y)>t\}.
\]
There is the exact countable-union identity
\begin{equation}\label{eq:Gunion}
 G_t=\bigcup_{j=1}^{\infty}F_{t+1/j}.
\end{equation}
Indeed, if $\tau_2(y)>t$, one can choose $j$ sufficiently large that
$t+1/j\leq\tau_2(y)$, including when $\tau_2(y)=+\infty$; the reverse
inclusion is immediate.

By \eqref{eq:tau-superlevel-upper}, for every $j\geq1$,
\[
 \dimH F_{t+1/j}\leq\frac1{t+1/j}<\frac1t=s.
\]
A set of Hausdorff dimension strictly below $s$ has zero
$s$-dimensional Hausdorff measure. Countable subadditivity and
\eqref{eq:Gunion} therefore give
\begin{equation}\label{eq:Gnull}
 \HH^s(G_t)=0.
\end{equation}
Also, $\mathbb{D}$ is countable, so $\HH^s(\mathbb{D})=0$ because $s>0$.

Every non-dyadic point of $W_t$ has exponent at least $t$ by
\eqref{eq:Wt}. It follows that
\[
 S_t:=(W_t\cap I)\setminus(G_t\cup\mathbb{D})\subseteq E_t\cap I.
\]
Removing a set of zero $\HH^s$-measure from a set of infinite
$\HH^s$-measure leaves infinite measure. Explicitly, subadditivity gives
\[
 \HH^s(W_t\cap I)
 \leq\HH^s(S_t)+\HH^s(G_t)+\HH^s(\mathbb{D})=\HH^s(S_t).
\]
Together with~\eqref{eq:criticalW}, this proves
\begin{equation}\label{eq:criticalE}
 \HH^{1/t}(E_t\cap I)=+\infty,\qquad(1<t<+\infty).
\end{equation}
Thus $\dimH(E_t\cap I)\geq1/t$. Combining this with
\eqref{eq:tau-superlevel-upper} yields the required equality for every finite
$t>1$.
 Although each
$F_{t+1/j}$ has dimension strictly less than $1/t$, the dimension of their
union can equal $1/t$. The removal is justified by their union having
\emph{zero $\HH^{1/t}$-measure}, not by claiming that the union has
strictly smaller dimension.

For $t=1$, the same level covers have total Lebesgue length
$O(2^{-(u-1)n})$ when $u>1$.  The Borel--Cantelli lemma gives
$\tau_2(y)\le1$ almost everywhere by taking a countable sequence
$u\rightarrow1^+$.  Since always $\delta_n(y)\le2^{-n-1}$ at a
non-dyadic point, $\tau_2(y)\ge1$.  Hence $E_1$ has full measure in
every $I$ and dimension one. \\
{\bf Step 4.}
The dimension of $E_\infty$ is at most $1/t$ for every $t>1$ by
\eqref{eq:tau-superlevel-upper}, and is therefore zero.
To see that it is uncountable in $I$, fix a binary cylinder with closure
inside $I$, and choose positions $N_j$ beyond its prefix such that
$N_{j+1}/N_j\to\infty$.  After the prefix, allow a digit equal to one
only at the positions $N_j$, with infinitely many of these digits chosen
to be one.  There are uncountably many such non-dyadic expansions, and
at level $N_j$ their tails satisfy
\[
 \delta_{N_j}(y)\le\sum_{i>j}2^{-N_i}
 \le2^{1-N_{j+1}}.
\]
Thus $\tau_2(y)=+\infty$ for all these points.  This also proves
nonemptiness, and completes the argument.
\end{proof}

\begin{proof}[Proof of Theorem~\ref{thm:holder-spectrum-main}]
Proposition~\ref{prop:pointwise-exponent} proves
\eqref{eq:pointwise-holder-main}.  If $0<h\le h_*$, its level set is
$E_{h_*/h}$, whose dimension is $h/h_*$ by
Lemma~\ref{lem:dyadic-spectrum}.  The level set for $h=0$ is
$E_\infty$, and no exponent exceeds $h_*$.  This proves
\eqref{eq:holder-spectrum-main}.  The full-measure, density, and
uncountability assertions follow from the intervalwise assertions in
the same lemma.
\end{proof}

\section{Explicit evaluations and concluding remarks}\label{sec:examples}

\subsection{Finite-index coordinates}

Formula~\eqref{eq:finite-derivative} makes the values at finite index
coordinates accessible to the usual MZV identities.  For example, at an
index $k\ge2$,
\begin{equation}\label{eq:depth-one-D}
 \Phi(2^{-k})=\zeta(k),
 \qquad D(2^{-k})=2^k(\zeta(k)-1).
\end{equation}
If $k\ge3$, the jump series can be summed in terms of ordinary zeta
values:
\begin{equation}\label{eq:depth-one-jump}
 \begin{split}
 D((2^{-k})^-)-D(2^{-k})
 &=2^k\bigl(\zeta(k-1)-2\zeta(k)+1\bigr),\\
 D((2^{-k})^-)&=2^k\bigl(\zeta(k-1)-\zeta(k)\bigr).
 \end{split}
\end{equation}
Indeed, expand $(m-2)m^{-k}=m^{-(k-1)}-2m^{-k}$ in
\eqref{eq:jump-main}; the $m=2$ term is zero.  At $k=2$ the first
series diverges, so
\[
 D(1/4)=4(\zeta(2)-1),\qquad D((1/4)^-)=+\infty.
\]

For an admissible index $(a,b)$, the coarsening identity gives
\begin{equation}\label{eq:depth-two-D}
 D\left(2^{-a}+2^{-(a+b)}\right)
 =2^{a+b}\bigl(\zeta^\star(a,b)-\zeta(a)\bigr).
\end{equation}
Thus the finite derivative values are rational linear combinations of
ordinary MZVs, with the normalization $2^{\wt(\mathbf k)}$.  The weights need not be homogeneous because deleting the
last entry changes the weight.

The exceptional dyadics also admit a uniform evaluation using the
classical MZV sum formula \cite{Granville1997}:
\begin{equation}\label{eq:classical-sum-formula}
 \sum_{\substack{k_1+\cdots+k_r=N\\k_1\ge2,\ k_2,\ldots,k_r\ge1}}
 \zeta(k_1,\ldots,k_r)=\zeta(N),\qquad 1\le r<N.
\end{equation}
The coarsenings of $(2,\{1\}^{r-1})$ run through all admissible indices
of weight $r+1$.  Summing \eqref{eq:classical-sum-formula} over their
depths and using \eqref{eq:coarsening} gives
\[
 \zeta^\star(2,\{1\}^{r-1})=r\zeta(r+1),\qquad r\ge1.
\]
Consequently, for every integer $M\ge3$,
\begin{equation}\label{eq:exceptional-dyadic-value}
 D(1/2-2^{-M})
 =2^M\bigl[(M-1)\zeta(M)-(M-2)\zeta(M-1)\bigr].
\end{equation}

\subsection{Infinite cyclic indices}

\begin{proposition}\label{prop:periodic-values}
For an integer $k\ge2$, put $y_k=(2^k-1)^{-1}$.  Then
\begin{equation}\label{eq:periodic-D}
 \begin{split}
 \Phi(y_k)&=\prod_{m=2}^{\infty}(1-m^{-k})^{-1},\\
 D(y_k)&=\prod_{m=3}^{\infty}\left(1-(2/m)^k\right)^{-1}.
 \end{split}
\end{equation}
In particular,
\begin{equation}\label{eq:one-third}
 \Phi(1/3)=2,\qquad D(1/3)=6, \qquad h_D(1/3)=h_*.
 \end{equation}
\end{proposition}

\begin{proof}
The binary expansion of $y_k$ is $(0^{k-1}1)^\infty$, so its infinite
index is $\{k\}^\infty$.  In a finite star sum with $r$ entries equal
to $k$, specify the multiplicity of each integer at least two, and let
ones fill the remaining positions.  Letting $r\to+\infty$ gives the
first product in \eqref{eq:periodic-D} by monotone convergence.

For the normalized tail $S(\{k\}^r)$, the analogous argument uses the
weights $(2/m)^k$ for $m\ge3$, and lets twos fill the remaining
positions.  Proposition~\ref{prop:derivative-input} gives the second
product.  Both products converge because the sums of their nontrivial
weights converge and each individual weight is less than one.
For $k=2$, the finite products are
\[
 \prod_{m=2}^N\frac{m^2}{m^2-1}=\frac{2N}{N+1},
 \qquad
 \prod_{m=3}^N\frac{m^2}{m^2-4}
 =\frac{6N(N-1)}{(N+1)(N+2)}.
\]
Their limits give the first two values in \eqref{eq:one-third}.
The
distance of the periodic binary expansion to the level $n$ dyadic rationals is bounded above and below by
positive constants times $2^{-n}$.  Hence $\tau_2(1/3)=1$, and the
H\"older formula gives the last value.
\end{proof}

\subsection{Scope of the method}

The four main results arise from three finite-index facts: the normalized
star-tail formula for $D$, the boundary recursion
\eqref{eq:alpha-recursion}, and the repeated-one harmonic-star generating
function \eqref{eq:all-one-generating}.  The passage to binary words does
not replace the MZV structure. It records the weight and the commas
of an index.  The coefficient at the smallest nonconstant summation
value, namely three after the contribution of twos has been removed,
accounts for the exponent $h_*$ and the graph dimension $2-h_*$.

For deformed or restricted-index correspondences, the same scheme would
require a replacement for the positive boundary recursion, a strict
separation between the first and remaining contraction rates, and a
dimension estimate for the resulting leading symbolic measure.  These
conditions are not verified here for the rational or complex variations
of the correspondence.  Likewise, the results above do not determine
Assouad dimensions, dimensions of fibers, or the absolute continuity of
the vertical distribution of $D$.  Such questions concern additional
structure beyond the graph and pointwise-regularity statements proved
in this paper.

\section*{Acknowledgements}
The author wants to thank Yufeng Wu for very helpful conversations about fractal geometry.
This project is  supported  by the National Natural Science Foundation of China (Grant No. 12571009) and the Natural Science Foundation of Hunan Province, China (Grant No. 2026JJ40003).    AI tools were used during the preparation of the paper as technical assistants.   All mathematical statements and proofs  were verified by the author.  The author takes full responsibility for the content of
the paper.

\end{document}